\documentclass[a4paper,11pt,reqno]{amsart}
\usepackage{a4wide}

\usepackage{amsmath,amssymb,amsthm,mathtools}

\usepackage[english]{babel}
\usepackage[colorlinks,citecolor=green,linkcolor=red]{hyperref}
\usepackage{amsthm}
\usepackage{color}
\usepackage{mathrsfs}
\usepackage{amsmath}
\usepackage{mathtools}
\usepackage{amssymb}
\usepackage{bm}
\usepackage{physics}
\usepackage{enumitem}
\usepackage{tikz-cd}
\usepackage{a4wide}
\usepackage{cleveref}
\usepackage{esint}
\usepackage{nicefrac}
\usepackage{stmaryrd}
\usepackage[margin=2cm]{geometry}
\usepackage{comment}
\usepackage{mathrsfs}
\usepackage{dutchcal}
\usepackage{xfrac}

\usepackage{yfonts}
\usepackage{stmaryrd}

\usepackage{listings}

\numberwithin{equation}{section}

\theoremstyle{plain}
\newtheorem{thm}{Theorem}[section]

\newtheorem{prop}[thm]{Proposition}

\newtheorem{lemma}[thm]{Lemma}
\newtheorem{theorem}[thm]{Theorem}
\newtheorem{corollary}[thm]{Corollary}

\theoremstyle{remark}

\newtheorem{rem}[thm]{Remark}

\theoremstyle{definition}
\newtheorem{defn}[thm]{Definition}

\DeclareMathOperator{\lip}{lip}
\DeclareMathOperator{\Ric}{Ric}

\newcommand{\de}{\ensuremath{\,\mathrm d}}

\let\oldchi=\chi
\renewcommand{\chi}{\text{\raisebox{\depth}{\(\oldchi\)}}}
\let\phi\varphi

\DeclareMathOperator{\dive}{div}
\DeclareMathOperator{\Lip}{Lip}

\DeclareMathOperator{\hess}{Hess}
\newcommand{\supp}{{\mathrm {supp\,}}}

\newcommand{\m}{\ensuremath{\mathfrak m}}
\newcommand{\mres}{\mathbin{\vrule height 1.6ex depth 0pt width 0.13ex\vrule height 0.13ex depth 0pt width 1.3ex}}

\newcommand{\Tan}{\mathrm {Tan}}

\newcommand{\RR}{\mathbb{R}}

\newcommand{\NN}{\mathbb{N}}

\newcommand{\LIP}{{\mathrm {LIP}}}

\newcommand{\RCD}{{\mathrm {RCD}}}

\def\Xint#1{\mathchoice
	{\XXint\displaystyle\textstyle{#1}}%
	{\XXint\textstyle\scriptstyle{#1}}%
	{\XXint\scriptstyle\scriptscriptstyle{#1}}%
	{\XXint\scriptscriptstyle\scriptscriptstyle{#1}}%
	\!\int}
\def\XXint#1#2#3{{\setbox0=\hbox{$#1{#2#3}{\int}$}
		\vcenter{\hbox{$#2#3$}}\kern-.5\wd0}}

\def\dashint{\Xint-}

\newcommand{\defeq}{\vcentcolon=}

\newcommand{\haus}{\mathscr{H}}
\newcommand{\X}{\mathsf{X}}

\renewcommand{\d}{{\rm d}}

\let\phi\varphi
\let\epsilon\varepsilon
\let\eps\varepsilon

\newcommand{\Test}{{\mathrm {Test}}}

\newcommand{\mass}{\m}

\author{Camillo Brena}
\address{\parbox{\linewidth}{ETH Z\"urich.\\
	R\"amistrasse 101,
    8092 Z\"urich, Switzerland\\[-4pt]\phantom{a}}}\email{camillo.brena@math.ethz.ch}
\author{Luca Gennaioli}
\address{\parbox{\linewidth}{Mathematics Institute, University of Warwick.\\ Coventry, CV4 7AL, United Kingdom\\[-4pt]\phantom{a}}}\email{luca.gennaioli@warwick.ac.uk}
\begin{document}
	\title[Dimension bound for unweighted RCD spaces]{Improvement of the dimension bound \\for unweighted RCD spaces}
	
\begin{abstract}
    We show that if $(\X,\d,\haus^n)$ is an $\RCD(K,N)$ space for some $K\in\RR$ and $N\in [1,\infty)$, then it is also a (non-collapsed) $\RCD(K,n)$ space. In other words, unweighted finite-dimensional $\RCD$ spaces enjoy an automatic improvement of the dimension bound to the Hausdorff dimension of the space. More generally, we prove that this holds also when $N=\infty$, provided that we assume in addition the $n$-rectifiability of the space.
\end{abstract}
	
	\maketitle
\setcounter{tocdepth}{4}

\section{Introduction}

The Curvature-Dimension condition, introduced in \cite{Lott-Villani09,Sturm06I,Sturm06II}, provides a synthetic formulation of lower bounds for the Ricci curvature and upper bounds for the dimension  on metric measure spaces. The curvature and dimension bounds should be considered jointly and may depend on each other. Moreover, it may  happen that the  dimension bound given by this condition  cannot match the intrinsic, geometric, dimension of the space.  
We begin by illustrating these phenomena on smooth spaces.

Let $(M^n,g,e^{-V}\haus^n)$ be a smooth, complete $n$-dimensional 
Riemannian manifold without boundary, endowed with weight $V\in C^\infty(M)$. 
We consider, for $N\in [1,\infty]$, the Bakry--Émery Ricci tensor 
\begin{equation}
\Ric_{N}\defeq
    \begin{cases}
        \Ric+\hess V-\frac{1}{N-n}\nabla V\otimes\nabla V\qquad&\text{if $N>n$},\\
        \Ric&\text{if $N=n$ and $V$ is constant},\\
        -\infty&\text{otherwise}.
    \end{cases}
\end{equation}
Notice that the borderline case $N=n$ corresponds to the degeneration of the first expression when $N\downarrow n$. In particular, it makes sense to split the cases where $V$ is constant or not. 

It is well known that $(M^n,g,e^{-V}\haus^n)$  satisfies the Riemannian Curvature-Dimension condition $\RCD(K,N)$ if and only if we have the lower bound $\Ric_N\ge Kg$. For background on  $\RCD$ spaces, we refer the reader to \cite{Villani2017, AmbICM, gigli2023giorgi} and references therein; here we only mention that the $\RCD(K,N)$ condition consists in coupling the Curvature-Dimension condition (with bounds $K$ and $N$) with the  infinitesimal Hilbertianity of the space (which means that its Sobolev space is a Hilbert space). 

As a particular instance of this characterization of the $\RCD(K,N)$ condition for smooth manifolds, we have the following two  complementary implications:
\begin{alignat}{5}
    &(M^n,g,e^{-V}\haus^n)&&\text{ is } \RCD(K,n)\qquad &&\Rightarrow\qquad&&\text{$V$ is constant},\label{impl1}\\
    &(M^n,g,\haus^n)&&\text{ is }\RCD(K,N)\qquad &&\Rightarrow\qquad&&\text{$(M^n,g,\haus^n)$  is $\RCD(K,n)$}. \label{impl2} 
\end{alignat}
To elaborate, \eqref{impl1} means that whenever the synthetic  dimension bound can be taken to be equal to the geometric dimension, then the weight must be constant. On the other hand, \eqref{impl2} means that, in the unweighted case, the synthetic  dimension bound can be taken to match the geometric one.
It is natural to ask whether \eqref{impl1} and  \eqref{impl2}  continue to hold when we replace the smooth  manifold $(M^n,g,e^{-V}\haus^n)$  with  a metric measure space $(\X,\d,\m)$. The generalization of  \eqref{impl1} to the nonsmooth setting was conjectured in \cite[Remark 1.11]{DPG17}  and later proved in \cite{H19,brena2021weakly}. On the other hand, the counterpart of \eqref{impl2} was conjectured in \cite[Conjecture 4.2]{Honda2020}.
The aim of the present paper is to give an affirmative answer to this question.

\begin{theorem}\label{main1}
    Let $(\X,\d,\haus^n)$ be an $\RCD(K, \infty)$ space. Assume that $(\X,\d,\haus^n) $ is $n$-rectifiable. 
    Then $(\X,\d,\haus^n)$ is  a non-collapsed $\RCD(K,n)$ space.
\end{theorem}
We remark that any $(\X,\d,\haus^n)$  $\RCD(K, N)$ space  with $N<\infty$ is $n$-rectifiable by \cite{Mondino-Naber14,bru2018constancy,  BruPasSem20}. In particular, the requirement that the space is $n$-rectifiable is not too restrictive, and would anyway be a consequence of the claim of the theorem. While our proof  uses the rectifiability of the space, it is not clear to us whether it is a necessary assumption or not (see Remark \ref{rectused}).
In any case, as a consequence of this discussion and Theorem \ref{main1}, we have the following unconditional result.
\begin{theorem}\label{main15}
    Let $(\X,\d,\haus^n)$ be an $\RCD(K, N)$ space with $N<\infty$.
    Then $(\X,\d,\haus^n)$ is  a non-collapsed $\RCD(K,n)$ space.
\end{theorem}
We point out that the result above allows one to exploit the  theory developed for non-collapsed $\RCD$ spaces, providing us with a wide range of applications. For instance, we have the following low-dimensional consequence:
Any finite-dimensional $\RCD$ space endowed with $\haus^2$ satisfies synthetic sectional curvature bounds, see \cite{LySta}.

Theorems \ref{main1} and \ref{main15} follow immediately from Theorem \ref{main2} below, by \cite{Gigli14}, as remarked in \cite[Theorem 2.22]{brena2021weakly}. We repeat the argument in Proposition \ref{trlapRCD} below for the reader's convenience.

\begin{theorem}\label{main2}    Let $(\X,\d,\haus^n)$ be an $\RCD(K, \infty)$ space. Assume that $(\X,\d,\haus^n) $ is $n$-rectifiable. 
    Then, for every $u\in\Test(\X)$,
    \begin{equation}
        \tr\hess u=\Delta u\qquad\haus^n\text{-a.e.}
    \end{equation}
\end{theorem}
We can rephrase the  results of the present paper  and  (some of the results) of \cite{bru2018constancy,H19, brena2021weakly} as follows.
\begin{corollary}
    Let $(\X,\d,\mass)$ be an $\RCD(K,N)$ space with $N<\infty$ and let $n\in\NN$ denote its essential dimension. Then, the following are equivalent.
    \begin{enumerate}
        \item $\mass= c\haus^n$ for some $c\in (0,\infty)$,
        \item $\tr\hess u=\Delta u\ \mass$-a.e.\ for every $u\in\Test(\X)$,
        \item $(\X,\d,\mass)$  is an $\RCD(K,n)$ space.
    \end{enumerate}
\end{corollary}
\begin{rem}\label{rectused}
In the course of the proof of Theorem \ref{main2}, the rectifiability assumption on     $(\X,\d,\haus^n)$ is used essentially only once, namely to apply Theorem \ref{dimtan}. Hence, it can be safely replaced by the assumption that $L^2(T\X)$ has constant dimension equal to $n$. Notice that $L^2(T\X)$ always has local dimension less than or equal to $n$ by \cite{ErikssonBique2024} or \cite{lucanicola}, so that the  issue is to show that the local dimension of $L^2(T\X)$ cannot be strictly smaller than $n$.

We also remark   that, if $T$ is the current defined in \eqref{defnt}, then we show in Proposition \ref{Tint} that $T\in \mathbf{N}_n(\X)$ without relying on  the rectifiability of $\X$. Hence, \cite[Theorem 8.7]{AmbrosioKirchheim00} implies that $E$ is automatically rectifiable   and then it is easy to obtain rectifiability of $\X$. In other words, the assumption that $L^2(T\X)$ has constant dimension equal to $n$ is itself enough to ensure rectifiability. We point out that this last fact is already known to experts, independently of our proof.
\end{rem}

Inspecting our argument, it is easy to realize that it is local in nature. In particular, we have the following corollary.
\begin{corollary}
Let $(\X,\d,\mass)$ be an $\RCD(K, \infty)$ space and let $U\subseteq\X$ be open. Assume that $\mass\,\mres U=\haus^n\,\mres U$ and that $(U,\d,\haus^n)$ is $n$-rectifiable.   Then, for every $u\in\Test(\X)$,
 \begin{equation}
        \tr\hess u=\Delta u\qquad\haus^n\text{-a.e.\ on $U$}.
    \end{equation}
Moreover, we have the following Bochner inequality on $U$. For any $f\in D(\Delta)$ with $\Delta f\in H^{1,2}(\X)$, 
        \begin{equation}  \label{bochner}
            \frac12 \int \Delta\varphi|Df|^2 \de \mass\ge \int \varphi\Big(\frac{(\Delta f)^2}{n}+\nabla f\,\cdot\,\nabla \Delta f+K|Df|^2\Big)\de\mass
        \end{equation}
        holds for any $\varphi\in D(\Delta)\cap L^\infty(\X)$ such that $\varphi\ge 0\ \mass$-a.e., $\Delta\varphi\in L^\infty$ and ${\rm dist}(\supp(\varphi), \X\setminus U)>0$.
\end{corollary}
\subsection*{Acknowledgments}
L.G.\ is supported by UK Research and Innovation (UKRI) under the Horizon Europe funding guarantee [grant number EP/Z000297/1].
The authors warmly thank Prof.\ L.\ Ambrosio, Prof.\ S.\ Honda and Prof.\ D.\ Semola  for helpful comments on an earlier draft of this manuscript.

\section{Preliminaries}
We now recall some preliminary notions along with some  new results that will be useful in the course of the proof.

\subsection{$\RCD$ spaces}
A metric measure space is a triplet $(\X,\d,\mass)$, where $(\X,\d)$ is a complete and separable metric space and $\mass$ is a Borel measure which is finite on bounded sets. For the general literature on $\RCD$ spaces we refer the reader to \cite{Villani2017, AmbICM, gigli2023giorgi}. 

We assume that the reader is familiar with the calculus developed in \cite{Gigli14}, see also \cite{GP19}; in particular, we make use of the  various function spaces defined there. We recall here the notions that will be used most frequently.

\begin{defn}
    Let $K\in\RR$ and $N\in [1,\infty]$. A metric measure space $(\X,\d,\mass)$ is said to be an $\RCD(K,N)$ space if the following conditions are satisfied.
    \begin{enumerate}
        \item For some $x\in\X$, there exists $C>1$ such that $\mass(B_r(x))\le C e^{Cr^2}$ for every $r>0$.
        \item $f\in L^2(\m)\mapsto \int |Df|^2\de\mass$ is a quadratic form. 
        \item Any $f\in H^{1,2}(\X)$ with $|Df|\le 1\ \mass$-a.e.\ has a $1$-Lipschitz representative.
        \item For any $f\in D(\Delta)$ with $\Delta f\in H^{1,2}(\X)$, 
        \begin{equation}  \label{bochner}
            \frac12 \int \Delta\varphi|Df|^2 \de \mass\ge \int \varphi\Big(\frac{(\Delta f)^2}{N}+\nabla f\,\cdot\,\nabla \Delta f+K|Df|^2\Big)\de\mass
        \end{equation}
        holds for any $\varphi\in D(\Delta)\cap L^\infty(\X)$ with $\varphi\ge 0\ \mass$-a.e.\ and $\Delta\varphi\in L^\infty$.
    \end{enumerate}
\end{defn}
As already remarked, the following result is well known. Nevertheless, we include the proof for the reader's convenience.
\begin{prop}\label{trlapRCD}
    Let $(\X,\d,\mass)$ be an $\RCD(K,\infty)$ space. Assume that $L^2(T\X)$ has constant dimension $n\in\NN$ and that for every $u\in\Test(\X)$, $\tr\hess u=\Delta u\ \mass$-a.e.
    Then,  $(\X,\d,\mass)$ is an $\RCD(K,n)$ space.
\end{prop}
\begin{proof}
    We know that \eqref{bochner} holds with $N=\infty$ and we have to show that it holds for $N=n$. Let $f\in\Test(\X)$. Also, let   $\varphi$  be Lipschitz with bounded support and satisfying $\varphi\ge 0\ \mass$-a.e. 
    By \cite[Theorem 3.3.8]{Gigli14} and the fact that the singular part of $\mathbf{\Gamma}_2$ is non-negative (see, e.g., \cite[Equation (3.1.10)]{Gigli14}, derived from \cite{Savare13}) we have that
    \begin{equation}
            -\frac12 \int \nabla\varphi \,\cdot\,\nabla|\nabla f|^2  \de \mass\ge \int \varphi\big(|\hess f|_{\mathrm{HS}}^2+\nabla f\,\cdot\,\nabla \Delta f+K|Df|^2\big)\de\mass.
    \end{equation}
    This means that we have improved \eqref{bochner} by adding the non-negative term $\varphi|\hess f|_{\mathrm{HS}}^2$ on the right-hand side. By Cauchy--Schwarz,
    \begin{equation}
        |\hess f|_{\mathrm{HS}}^2\ge \frac{(\tr\hess f)^2}{n}\qquad\mass\text{-a.e.},
    \end{equation}
    so that, recalling the  assumption $\tr\hess f=\Delta f\ \mass$-a.e., we have that 
\begin{equation}
            -\frac12 \int \nabla\varphi \,\cdot\,\nabla|\nabla f|^2  \de \mass\ge \int \varphi\Big(\frac{(\Delta f)^2}{n}+\nabla f\,\cdot\,\nabla \Delta f+K|Df|^2\Big)\de\mass.
    \end{equation}
    By approximation, the above continues to hold for every $\varphi \in H^{1,2}(\X)\cap L^\infty(\X)$ with $\varphi\ge 0\ \mass$-a.e. If moreover $\varphi\in D(\Delta)$, 
    \begin{equation}
            \frac12 \int \Delta\varphi |\nabla f|^2  \de \mass\ge \int \varphi\Big(\frac{(\Delta f)^2}{n}+\nabla f\,\cdot\,\nabla \Delta f+K|Df|^2\Big)\de\mass.
    \end{equation}
   The claim then follows by approximation on $f$ via heat flow, see  the proof of \cite[Corollary 3.3.9]{Gigli14}.
\end{proof}
\subsection{Calculus on $\RCD$ spaces}
For future reference, we state some results concerning calculus on $\RCD$ spaces. We remark that the relevant material about function spaces is taken from \cite{Gigli14}. In particular, $H$ spaces consist of the closures, with respect to the relevant norms, of test objects. For instance, $H^{1,2}_C(T\X)$ is the closure of test vector fields with respect to 
\begin{equation}
    X\mapsto \Big(\int |X|^2+|\nabla X|^2\de\mass\Big)^{1/2},
\end{equation} and $H^{1,2}_H(T\X)$ is the closure of test covector fields with respect to 
\begin{equation}
    X\mapsto \Big(\int |X|^2+|d X|^2+|\delta X|^2\de\mass\Big)^{1/2}.
\end{equation}
In this note, we are going to use the Riesz isomorphism to identify vector fields and covector fields. We will then identify also the exterior powers $\Lambda^kT \X$ and $\Lambda^k T^* \X$ throughout.

\begin{lemma}
    [{\cite[Lemma 4.3]{BGBV}}]\label{remaininH}
    Let $(\X,\d,\mass)$ be an $\RCD(K,\infty)$ space. Let $X\in H^{1,2}_H(T\X)\cap L^\infty(T\X)$ and let $f\in S^2(\X)\cap L^\infty(\X)$. Then $fX \in H^{1,2}_H(T\X)\cap L^\infty(T\X)$. 
\end{lemma}
\begin{lemma}[{\cite[Remark 2.3]{brena2023fine}}]\label{loc1} Let $(\X,\d,\mass)$ be an $\RCD(K,\infty)$ space.
    Let $X\in H^{1,2}_H(T\X)$. Then $\dive X=0$ $\mass$-a.e.\ on $\{X=0\}$.
\end{lemma}
\begin{lemma}\label{loc15} Let $(\X,\d,\mass)$ be an $\RCD(K,\infty)$ space.
    Let $X\in H^{1,2}_d(\Lambda^mT\X)$ with $|X|\in H
    ^{1,2}(\X)$. Then 
$d X=0$ $\mass$-a.e.\ on $\{X=0\}$.
\end{lemma}
\begin{proof}
   The proof is exactly  that of \cite[Remark 2.3]{brena2023fine}. We include it anyway for the reader’s convenience.
   Define 
   \begin{equation}
       \tilde\varphi_k(t)\defeq
       \begin{cases}
           1\qquad&\text{for }t\le 0,\\
           1-kt&\text{for }0<t<k^{-1},\\
           0&\text{for }t\ge k^{-1}.
       \end{cases}
   \end{equation}
   We also define $\varphi_k\defeq \tilde\varphi_k\circ|X|$. By the calculus rules, $\varphi_k\in H^{1,2}(\X)$ with 
   \begin{equation}
       |\nabla\varphi_k|=|\tilde\varphi_k'|\circ |X||\nabla|X||\le k\chi_{\{|X|\in (0,k^{-1})\}}|\nabla|X||,
   \end{equation}
   in particular, $|\nabla\varphi_k||X|\le\chi_{\{|X|\in (0,k^{-1})\}} |\nabla|X||\rightarrow 0$ in $L^2(\X)$.  It follows that $d(\varphi_kX)= d\varphi_k \wedge X+\varphi_k dX\rightarrow \chi_{\{|X|=0\}}dX$ in $L^2(\Lambda^{m+1}TX)$. Now we notice that $\varphi_k X\rightarrow 0$  in $L^2(\Lambda^mTX)$ and employ the closedness of the differential of \cite[Theorem 3.5.2]{Gigli14}.
\end{proof}
\begin{lemma}[{\cite[Proposition 3.4.9]{Gigli14}}]\label{loc2} Let $(\X,\d,\mass)$ be an $\RCD(K,\infty)$ space.
    Let $X\in H^{1,2}_C(T\X)$.   Then $\nabla X=0$ $\mass$-a.e.\ on $\{X=0\}$.
\end{lemma}

In the next proposition, the vector field $\alpha$ should be thought of as the Hodge star of the $(n-1)$-form $\eta\defeq f d\pi_1\wedge\cdots d\pi_{n-1}$, up to a sign that depends on $n$. In particular, \eqref{vefdscdc} is nothing more than the well-known formula $d\eta=(-1)^{n-1}\ast(\tr(\nabla (\ast \eta)))$. Due to the lack of regularity of the space, we have to prove it by a careful computation.
\begin{prop}\label{fdcsc} Let $(\X,\d,\mass)$ be an $\RCD(K,\infty)$ space, let $E$ be a Borel set and let $f,\pi_1,\dots,\pi_{n-1}\in\Test(\X)$.  Assume that $L^2(T\X)$ has constant dimension $n$ on $E$   and let $V_1,\dots,V_n\in H^{1,2}_H(T\X)\cap L^\infty(T\X)$ be such that, for $i,j=1,\dots,n$, $V_i\,\cdot\,V_j=\delta_{i,j} \ \mass$-a.e.\ on  $E$.
    Define  $\alpha$  as follows:
    \begin{equation}\label{defalpha}
        \alpha\defeq \sum_{i=1}^n (-1)^{i-1} \big(fd\pi_1\wedge\cdots\wedge d\pi_{n-1}\,\cdot\,V_1\wedge\cdots \wedge \hat V_i \wedge\cdots \wedge V_n\big)V_i.
    \end{equation}
    Then, $\alpha\in H^{1,2}_H( T\X)\cap L^\infty(T\X)$ and
    \begin{equation}\label{vefdscdc}
        d\big(f d\pi_1\wedge\cdots \wedge d\pi_{n-1}\big)=\tr(\nabla \alpha)V_1\wedge\cdots\wedge V_n\qquad\mass\text{-a.e.\ on }E.
    \end{equation}
\end{prop}
\begin{proof} 
First,  $\alpha\in H^{1,2}_H( T\X)\cap L^\infty(T\X)$ by Lemma \ref{remaininH}. 
To avoid the repetition of cumbersome notation, we  use the abbreviations:
\begin{alignat}{2}
 \omega&\defeq V_1\wedge\cdots \wedge V_n,
   \qquad
    &
    \hat\omega_{ i}&\defeq V_1\wedge\cdots \wedge \hat V_i \wedge\cdots \wedge V_n,\\
   \eta&\defeq f d\pi_1\wedge\cdots \wedge d\pi_{n-1}, \qquad
    &f_{ i}&\defeq \eta\,\cdot\,\hat \omega_{i}.
\end{alignat} 
We note that $(\hat \omega_{ i})_{i=1,\dots,n}$ forms a basis of $L^2(\Lambda^{n-1}T\X)$ on $E$. 
We can also define 
\begin{equation}
    \tilde \eta\defeq \sum_{i=1}^n f_i \hat\omega_{ i}=\sum_{i=1}^n \eta\,\cdot\,\hat\omega_i \hat\omega_{ i},\qquad
    \tilde\alpha\defeq \sum_{i=1}^n (-1)^{i-1} f_i V_i.
\end{equation}
Notice that  $\tilde\alpha\in H^{1,2}_H( T\X)\cap L^\infty(T\X)$ by Lemma \ref{remaininH}.

By $V_i\,\cdot\, V_j=\delta_{i,j}\ \mass$-a.e.\ on $E$, we see that $\tilde\eta=\eta\ \mass$-a.e.\ on $E$. Moreover, $\tilde\alpha=\alpha\  \mass$-a.e.\ on $E$. 
Now notice that, if $X\in L^2(T\X)$,
\begin{equation}
    \nabla |\eta-\tilde\eta|^2\,\cdot\,X=2\big(\nabla_X (\eta-\tilde\eta)\big)\,\cdot\, (\eta-\tilde\eta),
\end{equation}
where $\nabla_X(\eta-\tilde\eta)$ is defined as follows: For $g\in H^{1,2}(\X)\cap L^\infty(\X)$ and $X_1,\dots,X_{n-1}\in H^{1,2}_H(TX)\cap L^\infty(T\X)$,
\begin{equation}
    \nabla_X(g X_1\wedge\cdots\wedge X_{n-1})=\nabla_X g\wedge X_1\wedge\cdots\wedge X_{n-1}+ \sum_{i=1}^{n-1}g X_1\wedge\cdots\wedge\nabla X_i(X,\,\cdot\,)\wedge\cdots \wedge X_{n-1},
\end{equation}
and the definition is then extended by linearity. It follows that $|\eta-\tilde\eta|\in H^{1,2}(\X)$, so that, by Lemma \ref{loc15}, $d\tilde\eta=d\eta\ \mass$-a.e.\ on $E$ (that $\tilde\eta\in H^{1,2}_d(\Lambda^{n-1}T\X)$  is easily verified by closedness of the differential of \cite[Theorem 3.5.2]{Gigli14} and the Leibniz rule of \cite[Proposition 3.5.4]{Gigli14}).
Moreover, Lemma \ref{loc2} implies that $\nabla \tilde\alpha=\nabla\alpha \ \mass$-a.e.\ on $E$. Hence, it is enough to prove 
\begin{equation}\label{vefdscdc1}d\tilde\eta=\tr(\nabla\tilde\alpha)    \omega\qquad\mass\text{-a.e.\ on }E.
\end{equation}

All the following identities are understood $\mass$-a.e.\ on $E$ and all  sums are over $1,\dots,n$. First,
\begin{equation}
       d\tilde\eta=\sum_i  d f_i\wedge\hat \omega_{i}+ f_id\hat \omega_{ i}=\sum_i (-1)^{i-1}\nabla f_i\,\cdot\,V_i \omega+f_id\hat \omega_{i}.
\end{equation}
Further, it is easy to compute
\begin{align}
   d\hat\omega_i=(-1)^{i-1}\sum_{j\ne i}dV_j (V_i,V_j)\omega.
\end{align}
Now notice that if $g,h\in\Test(\X)$, then $d(gdh)(A,B)=\nabla (g\nabla h)(A,B)-\nabla (g\nabla h)(B,A)$ for every $A,B\in H^{1,2}_H(T\X)$. By density, this continues to hold with $gdh$ replaced by $V\in H^{1,2}_H(T\X)\cap L^\infty(T\X) $, where we  use our identification $TX=T^*X$ for $V$. It follows that
\begin{align}
       d\hat\omega_i&= (-1)^{i-1}\sum_{j\ne i}\Big(\nabla V_j (V_i,V_j)-\nabla V_j (V_j,V_i)\Big)\omega=-(-1)^{i-1}\sum_{j}\nabla V_j (V_j,V_i)\omega\\
     &=(-1)^{i-1}\sum_{j}\nabla V_i (V_j,V_j)\omega
\end{align}
where we used that $\nabla V_j(V_i,V_j)=0$  and $\nabla V_j (V_j,V_i)=-\nabla V_i (V_j,V_j)$ by \cite[Proposition 3.4.6]{Gigli14}. It follows that
\begin{equation}
     d\tilde\eta=\sum_i(-1)^{i-1} \nabla f_i\,\cdot\,V_i\omega+\sum_{i,j}(-1)^{i-1}f_i\nabla V_i(V_j,V_j)\omega.
\end{equation}
On the other hand,
\begin{align}
        \tr(\nabla\tilde\alpha)&=\sum_j \nabla \tilde\alpha(V_j,V_j)=\sum_j \sum_i (-1)^{i-1} (\nabla f_i\otimes V_i)(V_j,V_j)+ (-1)^{i-1} (f_i\nabla V_i)(V_j,V_j)\\
    &=\sum_i(-1)^{i-1} \nabla f_i\,\cdot\, V_i+\sum_j\sum_i(-1)^{i-1} f_i\nabla V_i(V_j,V_j).
\end{align}
Hence \eqref{vefdscdc1} is proved.\end{proof}

The role of the following  lemma is to show that perimeter measures are concentrated on $\mass$-null sets. Notice that in the setting of doubling spaces satisfying a Poincaré inequality, a  more precise description is available in \cite{amb01,amb00}. For finite-dimensional $\RCD$ spaces, see \cite{bru2019rectifiability,ABPrank}.
\begin{lemma}\label{persing}
    Let $(\X,\d,\m)$ be an $\RCD(K,\infty)$ space. Let $E$ be a set of finite perimeter. Then $|D\chi_E|\perp\m$, in the sense that the measure $|D\chi_E|$ is concentrated on an $\m$-null subset of $\X$. 
\end{lemma}
\begin{proof}
First, we recall  that, for fixed $x\in\X$, $B_r(x)$ has finite perimeter for a.e.\ $r$, thanks to the coarea formula. Hence, we can assume with no loss of generality  that $\m(E)<+\infty$. For $\eps\in(0,1/2)$, we consider the maps $T_\eps:\mathbb R\to\mathbb R$ defined as 
    \begin{equation*}
        T_\eps(s)\defeq\begin{cases}
            0\qquad &s\leq\eps, \\
            \frac{s-\eps}{1-2\eps}\quad &\eps<s<1-\eps, \\
            1\qquad &s\geq 1-\eps.
        \end{cases}
    \end{equation*}
  Observe that $T_\eps$ is Lipschitz and that $\Lip(T_\eps)=1/(1-2\eps)$ for all $\eps\in(0,1/2)$. Moreover $\chi_E=T_\eps(\chi_E)$ and, setting $u_t\defeq{\rm h}_t\chi_E$, we have $T_\eps(u_t)\to\chi_E$ in $L^1(\X)$ as $t\downarrow 0$ for all $\eps>0$. Thanks to the lower semicontinuity of the total variation and to the chain rule we get 
  \begin{equation}
  \label{giancarlo}
    |D\chi_E|(\X)\leq\liminf_{t\downarrow 0}\int|\nabla T_\eps(u_t)|\de\mass\leq\frac{1}{1-2\eps}\liminf_{t\downarrow 0}\int_{A_{t,\eps}}|\nabla u_t|\de\m,
  \end{equation}
  where we have defined $A_{t,\eps}\defeq\{\eps<u_t<1-\eps\}$. We now recall the Bakry--Émery estimate for $\chi_E\in{\rm BV}(\X)$ (see for example \cite[Proposition 4.1]{BGBV} after \cite{Savare13}), which reads as 
  \begin{equation}
      \label{eq:Bakry_Emery_RCD}
      |\nabla u_t|\m\leq e^{-Kt}{\rm h}_t|D\chi_E|=e^{-Kt}\big({\rm h}_t g\de\m+{\rm h}_t|D\chi_E|^s\big),
  \end{equation}
  where $|D\chi_E|=g\de\m+|D\chi_E|^s$ is the Lebesgue decomposition with $g\in L^1(\X)$ and $|D\chi_E|^s\perp\mass$.
Using \eqref{eq:Bakry_Emery_RCD} in \eqref{giancarlo} and exploiting linearity of the heat flow we get 
 \begin{equation}
     |D\chi_E|(\X)\leq\frac{1}{1-2\eps}\liminf_{t\downarrow 0}e^{-Kt}\Big(\int_{A_{t,\eps}}{\rm h}_tg\de\m+\int_{A_{t,\eps}}\d{\rm h}_t|D\chi_E|^s\Big).
 \end{equation}
 Observe now that the first term on the right-hand side goes to zero because ${\rm h}_tg\rightarrow g$ in $L^1(\X)$ and $\m(A_{t,\eps})\to 0$ as $t\downarrow 0$. Moreover,  we have ${\rm h}_t|D\chi_E|^s(A_{t,\eps})\leq{\rm h}_t|D\chi_E|^s(\X)=|D\chi_E|^s(\X)$, so that 
 \begin{equation}
    |D\chi_E|(\X)\leq\frac{1}{1-2\eps}|D\chi_E|^s(\X).
 \end{equation}
 Letting $\eps\downarrow 0$ finally yields
 \begin{equation}
    |D\chi_E|(\X)\leq|D\chi_E|^s(\X), 
 \end{equation}
 which proves the claim, since the absolutely continuous part of $|D\chi_E|$ has  to vanish by the above.
\end{proof}

\subsection{Essential dimension}\label{essdim}
The goal of this section is to show that $n$-rectifiable $\RCD$ spaces have a tangent module of constant dimension $n$. The main result is Theorem \ref{dimtan} below, which holds actually in a more general setting: $n$-rectifiable spaces supporting  a Poincaré inequality.
For finite-dimensional $\RCD$ spaces, this conclusion is already known by \cite{bru2021constancy, GP16}. The result for possibly infinite-dimensional $n$-rectifiable $\RCD$ spaces  is, to our knowledge, new. Our proof pivots on the theory of Lipschitz differentiability spaces (see the references in \cite{IPSabs}) and the results of \cite{IPSabs}.
Moreover, the technology recalled below will be a key ingredient in the proof of Proposition \ref{maincurrent} in the next section.

    \bigskip

 Let $(\X,\d,\mass)$ be a metric measure  space. A  weak chart of dimension $n$ is  a pair $(U,\varphi)$, where $U\subseteq\X$ is Borel with $\mass(U)>0$ and $\varphi:\X\rightarrow\RR^n$ is Lipschitz,  satisfying the following. For every $f\in\LIP(\X)$ and $\mass$-a.e.\ $x\in U$, there exists a unique linear map $d_xf:\RR^n\rightarrow\RR$ satisfying
\begin{equation}
    \lip(f_{|U}-d_xf \circ\phi_{|U})(x)=0.\label{dewfvsdc}
\end{equation}
Here, the local Lipschitz constant  is defined
\begin{equation}
    \lip(g)(x)\defeq \limsup_{y\rightarrow x}\frac{|g(y)-g(x)|}{\d(x,y)}
\end{equation}
at accumulation points  and is set to be zero at isolated points.
A metric measure space is called a weak Lipschitz differentiability space if it can be covered, up to an $\mass$-null subset, by a countable collection of weak charts $\mathcal{A}=(U_k,\varphi_k)_k$. In this case, following \cite{IPSabs}, we consider the concrete tangent bundle $T_{\mathcal A}\X\defeq \bigsqcup_{k} U_k\times\RR^{n_k}$, where $n_k$ is the dimension of the chart $(U_k,\varphi_k)$. $T_{\mathcal A}\X$ is equipped with the following norm: for $\mass$-a.e.\ $x\in U_k$, $\{x\}\times \RR^{n_k}$ is endowed with the dual norm of $(\RR^{n_k})^*\ni L\mapsto \lip({L\circ \varphi_k}_{|U_k})$. The space of $2$-integrable sections is independent of the chosen charts (up to isomorphism), and is denoted by $\Gamma_2(T\X)$, see \cite{IPSabs}.

\begin{theorem}\label{dimtan}
Let $(\X,\d,\haus^n)$ be an $n$-rectifiable metric measure space supporting a Poincaré inequality, namely, for every $f\in\LIP(\X)$, $x\in\X$ and $r\in (0,1)$, it holds
\begin{equation}\label{poincare}
    \int_{B_r(x)}\Big|f(y)-\dashint_{B_r(x)}f(z)\de\haus^n(z)\Big|\de\haus^n(y)\le C_{PI} r\int_{B_{2r}(x)}|D f|(y)\de\haus^n(y).
\end{equation}
Then  $(\X,\d,\haus^n)$  is a weak Lipschitz differentiability space with charts of constant dimension $n$. Moreover,  there exists an isometric isomorphism of normed modules $L^2(T\X)\simeq \Gamma_2(T\X)$.
In particular $L^2(T\X)$ has constant dimension $n$.
\end{theorem}
\begin{proof}
First, we show that $(\X,\d,\haus^n)$  is a weak Lipschitz differentiability space  with weak charts of dimension $n$. It is a classical result that we can find a sequence $(U_j,\varphi_j)_{j\in\NN}$, where $U_j$ are Borel subsets of $\X$, with $\haus^n(U_j)>0$ and  $\haus^n(\X\setminus\bigcup_j U_j)=0$  and $\varphi_j:U_j\rightarrow\RR^n$ are bi-Lipschitz onto their images. This follows from  \cite{Kir94} and will be recalled also in the first part of Section \ref{metrcurr}. Set $B_j\defeq\varphi_j(U_j)\subseteq\RR^n$. Fix $j$ and suppress the subscripts; we show that $(U,\varphi)$ is a weak chart. First, define $g\defeq f\circ\phi^{-1}\in\LIP( B)$. By extension, we consider $g\in\LIP(\RR^n)$. For a.e.\ $x'\in B$,  $g$ is differentiable, so that we can consider $\nabla g(x')$. Fix such a point $x'$ and set  $x\defeq \varphi^{-1}(x')$.
We have that
\begin{equation}\label{vefdcsc}
    \lim_{B\ni y'\rightarrow x'}\frac{g(y')-g(x')-\nabla g(x')(y'-x')}{|y'-x'|}=0.
\end{equation}
Hence, since $\varphi$ is bi-Lipschitz,  setting $y=\varphi^{-1}(y')$,
\begin{equation}\label{vtrefdssv}
        \lim_{U\ni y\rightarrow x}\frac{f(y)-f(x)-\nabla g(\phi(x))(\phi(y)-\phi(x))}{\d(y,x)}=0.
\end{equation}
Thus,  \eqref{dewfvsdc} (with $d_xf=\nabla g(\varphi(x))$) follows from  \eqref{vtrefdssv}.
 Uniqueness of $d_xf$ follows at density points of $B$ from the two equations above, by the uniqueness of $\nabla g$.  Of course, we can extend the chart $\varphi:U\rightarrow \RR^n$ to $\varphi:\X\rightarrow \RR^n$ Lipschitz.

If we show that there exists a family of moduli of continuity $(\omega_x)_{x\in U}$ such that, for every $f\in\LIP(\X)$, 
\begin{equation}\label{adscscdsc}
    \lip(f_{|U})(x)\le \omega_x(|D f|(x))\qquad\text{for $\haus^n$-a.e.\ $x\in U$},
\end{equation}
the conclusion  follows from \cite[Theorem 1.3]{IPSabs}.  We are going to do that directly; we remark, however, that another proof is already known, see  the argument in \cite[Proposition 5.1]{IPSabs}, after \cite{BateLiDiff}. 

First, by \cite[Theorem 9]{Kir94}, as stated in \cite[Theorem 5.4]{AmbKir00}, it holds that 
\begin{equation}\label{fdsac}
    \lim_{r\downarrow 0}\frac{\haus^n(B_{r}(x))}{\omega_nr^n}=1\qquad\text{for $\haus^n$-a.e.\ $x\in\X$}.
\end{equation}
Then, by the Poincaré inequality \eqref{poincare}, 
\begin{align}
    &\frac{1}{\haus^n(B_r(x))^2}\int_{B_r(x)\cap U}\Big| \int_{B_r(x)\cap U}\big(f(y)-f(z)\big)\de\haus^n(z)\Big| \de\haus^n(y)\\
    &\qquad\le\frac{1}{\haus^n(B_r(x))^2}\int_{B_r(x)\cap U}\Big| \int_{B_r(x)}\big(f(y)-f(z)\big) \de\haus^n(z)\Big|\de\haus^n(y)\\
    &\qquad\qquad+\frac{1}{\haus^n(B_r(x))^2}\int_{B_r(x)\cap U}\Big| \int_{B_r(x)\setminus U}\big(f(y)-f(z)\big) \de\haus^n(z)\Big|\de\haus^n(y)\\
    &\qquad\le C_{PI} \frac{r}{\haus^n(B_r(x))}\int_{B_{2r}(x)} |D f|(y)\de\haus^n(y) +\LIP(f) 2r\frac{\haus^n(B_r(x)\setminus U)}{\haus^n(B_r(x))}.
\end{align}
Now recall that $\varphi$ is bi-Lipschitz, say $L$-bi-Lipschitz; hence, for a constant $C_L$, $\varphi_*(\haus^n\mres U)=\theta_\varphi \mathscr{L}^n\mres B$, where $\theta_\varphi\in (C_L^{-1},C_L)$ a.e. We can thus use the above and obtain
\begin{align}
 & \frac{1}{C_L\haus^n(B_r(x))^2}\int_{B_{r/L}(x')\cap B}\Big| \nabla g(x')\,\cdot\,\int_{\varphi(B_{r}(x)\cap U)}\frac{ y'-z'}{r}\theta_\varphi(z')\de z'\Big| \de y'\\
   &\qquad\le  \frac{1}{\haus^n(B_r(x))^2}\int_{\varphi(B_{r}(x)\cap U)}\Big| \int_{\varphi(B_{r}(x)\cap U)}\frac{\nabla g(x')(y'-z')}{r}\theta_\varphi(z')\de z'\Big| \theta_\varphi (y')\de y'\\
    &\qquad\le \frac{C_L^2\haus^n(\varphi(B_{r}(x)\cap U))^2}{\haus^n(B_r(x))^2}\sup_{z',y'\in \varphi(B_{r}(x)\cap U)} \Big|\frac{g(y')-g(z')-\nabla g(x')(y'-z')}{r}\Big| \\
    &\qquad\qquad+ \frac{1}{\haus^n(B_r(x))^2}\int_{\varphi(B_{r}(x)\cap U)}\Big| \int_{\varphi(B_{r}(x)\cap U)}\frac{g(y')-g(z')}{r}\theta_\varphi(z')\de z' \Big| \theta_\varphi(y') \de y' \\
    &\qquad\le \frac{C_L^2\haus^n(\varphi(B_{r}(x)\cap U))^2}{\haus^n(B_r(x))^2}\sup_{z',y'\in \varphi(B_{r}(x)\cap U)} \Big|\frac{g(y')-g(z')-\nabla g(x')(y'-z')}{r}\Big| \\
    &\qquad\qquad+ C_{PI} \frac{1}{\haus^n(B_r(x))}\int_{B_{2r}(x)} |D f|(y)\de\haus^n(y) +2\LIP(f) \frac{\haus^n(B_r(x)\setminus U)}{\haus^n(B_r(x))}.
\end{align}
We recall \eqref{fdsac} and \eqref{vefdcsc} and we let  $r\downarrow 0$ and we see that, at a.e.\ $x\in U$ (and thus a.e.\ $x'=\varphi(x)\in B$, in particular, density point),
\begin{equation}
    \limsup_{r\downarrow 0}\frac{1}{(\omega_n r^n)^2}\int_{B_{r/L}(x')}\Big| \nabla g(x')\,\cdot\,\int_{\varphi(B_{r}(x)\cap U)}\frac{ y'-z'}{r}\theta_\varphi(z')\de z'\Big| \de y'\le C |D f|(x),
\end{equation}
for a constant $C$ depending only on $C_{PI},C_L$, whose value may change from line to line. We study the left-hand side. We change variables and rewrite 
\begin{equation}
    \limsup_{r\downarrow 0}\dashint_{B_{1/L}(0)}\Big| \nabla g(x')\,\cdot\,\frac{1}{\omega_n r^n}\int_{\varphi(B_{r}(x)\cap U)}\Big(y'-\frac{ z'-x'}{r}\Big)\theta_\varphi (z')\de z'\Big| \de y'\le C |D f|(x).
\end{equation}
Hence, for some vector $v'\in\RR^n$, with $|v'|\le C$,
we have that 
\begin{equation}
    \dashint_{B_{1/L}(0)}|\nabla g(x')\,\cdot\, (y'-v')|\de y'\le C |D f|(x).
\end{equation}
Thus, 
\begin{equation}
    |\nabla g|(x')\le C|D f|(x).
\end{equation}
It is easy to see from \eqref{vtrefdssv}  that, at a.e.\ $x\in U$, $\lip(f_{|U})(x)\le C_L|\nabla g(\varphi (x))|$, which, combined with the inequality above, proves \eqref{adscscdsc}.
%
\end{proof}

\subsection{Metric currents}\label{metrcurr}
We refer to \cite{AmbrosioKirchheim00} for the relevant notions about Ambrosio--Kirchheim metric currents and recall here those that will be most useful to us. 

We remark that in \cite{AmbrosioKirchheim00} and in the present paper, metric currents have, by definition, finite mass. We write $\mathbf{N}_k$ for the space of $k$-dimensional normal currents, i.e., those whose boundary is still a current (in the sense that it has finite mass). A $k$-dimensional metric current $T$ is called rectifiable if $|T|$ (its mass measure) is concentrated on a $k$-rectifiable set and vanishes on $\haus^k$-null sets, and is called integer-rectifiable if it also satisfies the following: For every $\phi:\X\rightarrow\RR^k$ Lipschitz and $A\subseteq \X$ open, we have that the push-forward $\phi_*(T\mres A)$ coincides with the current induced by  some $\theta\in L^1(\RR^k,\mathbb{Z})$. The space of $k$-dimensional integral metric currents, $\mathbf{I}_k$, is the space of integer-rectifiable normal $k$-currents.

We begin by recalling the Ambrosio--Kirchheim boundary-rectifiability theorem, \cite[Theorem 8.6]{AmbrosioKirchheim00}.
\begin{theorem}\label{boundrect} Let $\X$ be a complete metric space. Let $T\in\mathbf{I}_k(\X)$. Then $\partial T\in\mathbf{I}_{k-1}(\X)$.
\end{theorem}

We now fix some  notation for an $n$-rectifiable metric measure space $(\X,\d,\haus^n)$. Rectifiability means that there exist $(f_i)_{i\in\NN}$ Lipschitz with $f_i:K_i\rightarrow \X$, where each $K_i\subseteq\RR^n$ is a compact and $\haus^n\big(\X\setminus\bigcup_i f_i(K_i)\big)=0$.  By \cite[Lemma 4]{Kir94} (and the second paragraph of the proof of \cite[Theorem 7]{Kir94}), there is no loss of generality in assuming that the $f_i$ are bi-Lipschitz onto their images. Now fix $i$, i.e., consider $f:K\rightarrow\X$ bi-Lipschitz onto its image.

 It is well known, via Kuratowski embedding, that $(\X,\d)$ isometrically embeds into $\ell^\infty$, which is a $w^*$-separable dual space. In what follows we shall always embed the space $(\X,\d)$ into $\ell^\infty$ using such an embedding. Notice that in this case we can assume that $f$ is defined on $\RR^n$.
It is proved in \cite[Theorem 2]{Kir94} and \cite[Theorem 3.5]{AmbKir00} that for a.e.\ $x\in K$, there exist a linear map $wd_xf:\RR^n\rightarrow \ell^\infty $ and a seminorm $md_xf:\RR^n\rightarrow \RR$ (called metric differential) such that  
\begin{align}
    w^*\lim_{y\rightarrow x}\frac{f(y)-f(x)-wd_xf(y-x)}{|y-x|}=0,\\
    \lim_{|y-x|+|z-x|\rightarrow 0}\frac{\big|\d(f(y),f(z))-md_x f(y-z)\big|}{|y-x|+|z-x|}=0,\label{second}
    \end{align}
and satisfying
\begin{equation}
        md_x f(v)=\| wd_x f(v)\|_{\ell^\infty}\qquad\text{for every }v\in\RR^n.
\end{equation}
For a.e.\ $x\in K$, \cite{Kir94} defines
\begin{equation}
    \Tan^{(n)}(\X,f(x))\defeq wd_x f(\RR^n)
\end{equation}
and proves that this definition is well posed. The space $\Tan^{(n)}(\X,f(x))$ is naturally endowed with the norm induced by $md_x f$ and  with the corresponding Hausdorff  measure. During the proof of \cite[Theorem 7]{Kir94}, it is proved that if $B\subseteq K$ is a Borel set such that $md_x f$ is a seminorm but not a norm for a.e.\ $x\in B$, then $\haus^n(f(B))=0$. It follows that we can assume with no loss of generality that $md_x f$ is a norm for a.e.\ $x\in K$.

We now introduce  a simple yet useful result.
\begin{prop}\label{maincurrent}
Let $(\X,\d,\haus^n)$ be an $n$-rectifiable metric measure space supporting a Poincaré inequality, see \eqref{poincare}. Assume moreover that it is infinitesimally Hilbertian.
    Let $T$ be an $n$-dimensional  metric current on $(\X,\d)$ satisfying $|T|=\theta \haus^n$, where $\theta:\X\rightarrow\mathbb{N}$ is a Borel function. Then $T$ is integer-rectifiable. 
\end{prop}
\begin{proof}
We use the terminology of Section \ref{essdim} and we recall that Theorem \ref{dimtan} applies.
Let $(U,\varphi)$ be any weak chart,  where $\varphi:U\rightarrow K\subseteq\RR^n$ is bi-Lipschitz. To be more precise, $\varphi:U\rightarrow\RR^n$ is in principle only Lipschitz, but one can use the short argument at the beginning of the proof of  Theorem \ref{dimtan} to have this seemingly stronger characterization. Our first step is to prove that the metric differential of $\varphi^{-1}$ is  a.e.\ induced by a scalar product, see \eqref{toshow}.

Let now $L:\RR^n\rightarrow\RR$ be linear and consider, for $\haus^n$-a.e.\ $x\in U$, where $y=\varphi(x)$,
\begin{align}
    \lip({L\circ \varphi}_{|U})(x)&=\limsup_{U\ni x'\rightarrow x}\frac{|L(\varphi(x'))-L(\varphi(x))|}{\d(x',x)}=\limsup_{K\ni y'\rightarrow y}\frac{|L(y')-L(y)|}{\d(\varphi^{-1}(y'),\varphi^{-1}(y))}\\
    &=\limsup_{K\ni y'\rightarrow y}\frac{|L(y')-L(y)|}{|y'-y|}\frac{|y'-y|}{\d(\varphi^{-1}(y'),\varphi^{-1}(y))}.
\end{align}
Thanks to \eqref{second}, we see that at $\haus^n$-a.e.\ $x\in U$,
\begin{equation}
    \lip({L\circ \varphi}_{|U})(x)=\limsup_{K\ni y'\rightarrow y}\frac{|L(y')-L(y)|}{|y'-y|}\frac{|y'-y|}{md_y(\varphi^{-1})(y'-y)},
\end{equation}
where we used that $\varphi$ is bi-Lipschitz and that the metric differential of $\varphi^{-1}$ is a norm at $y$. At density points of $K$, we thus see that 
\begin{equation}
    \lip({L\circ \varphi}_{|U})(x)=\sup_{v\in\RR^n\setminus\{0\}}\frac{|L(v)|}{md_y(\varphi^{-1})(v)}.
\end{equation}
By definition, $T_{\mathcal{A}}\X$ is endowed with the dual norm of $L\mapsto \lip({L\circ\varphi}_{|U})$ (for $x\in U$), which is precisely $md_{y}(\varphi^{-1})$, by the above equation. By Theorem \ref{dimtan}, $\Gamma_2(T\X)\simeq L^2(T\X)$, which is a Hilbert module by the infinitesimal Hilbertianity assumption. 
Thus, we have shown that
\begin{equation}\label{toshow}
    md_y(\varphi^{-1})\text{ is induced by a scalar product for a.e.\ $y\in K$}.
\end{equation}

We now recall the area factor considered in \cite[Equation (9.11)]{AmbrosioKirchheim00}: If $V$ is an $n$-dimensional Banach space, then
\begin{equation}
    \lambda_V\defeq \frac{2^n}{\omega_n}\sup\Big\{\frac{\haus^n(B_1)}{\haus^n(R)}:V\supseteq R\supseteq B_1,\text{ $R$ parallelepiped}\Big\},
\end{equation}
where $B_1$ is the unit ball of $V$ and $R$, the parallelepiped, is any image of the Euclidean cube through a linear map $\RR^n\rightarrow V$. We will be interested only in the case where $V=\Tan^{(n)}(\X,x)$, endowed with the norm induced by the metric differential.
We know from \cite[Lemma 9.2]{AmbrosioKirchheim00} that $\lambda_V=1$ whenever the norm on $V$ is induced by a scalar product. Hence, by \eqref{toshow},
\begin{equation}\label{cedscdcscd}
   \lambda(x)\defeq \lambda_{\Tan^{(n)}(\X,x)}=1\qquad\text{for $\haus^n$-a.e.\ $x\in\X$}.
\end{equation}

By the rectifiability of the space and the assumption on $|T|$, we can apply \cite[Theorem 9.1]{AmbrosioKirchheim00} to deduce that $T$ is induced by a triplet $(S,\theta',\tau)$, where $S$ is $n$-rectifiable, $\theta':S\rightarrow (0,\infty)$ is Borel and $\tau$ is an orientation of $S$ (see \cite[Section 9]{AmbrosioKirchheim00}). Applying  \cite[Theorem 9.5]{AmbrosioKirchheim00} with \eqref{cedscdcscd}, we see that $|T|=\theta'\lambda\haus^n\mres S=\theta'\haus^n\mres S$.  Recalling the assumption on $|T|$, we deduce that $\theta'=\theta\ \haus^n$-a.e. Now, inspecting the proof of \cite[Theorem 9.1]{AmbrosioKirchheim00} (see in particular \cite[Equation (9.10)]{AmbrosioKirchheim00}), we see that $T$ is integer-rectifiable by \cite[Theorem 4.5]{AmbrosioKirchheim00}.
\end{proof}

\section{Proof of Theorem \ref{main2}}
In this section, we fix a space $(\X,\d,\haus^n)$ as in the statement of Theorem \ref{main2}. In particular, $(\X,\d,\haus^n)$  satisfies a Poincaré inequality, see \eqref{poincare},   thanks to \cite[Theorem 1.2]{Rajala12-2} (see also \cite[Theorem 1]{Rajala12}).
Hence, by Theorem \ref{dimtan}, we know that $L^2(T\X)$ has constant dimension $n$. Our key tool is the following.

\begin{thm}\label{intermediate}
Let $(\X,\d,\haus^n)$ be as in the statement of Theorem \ref{main2}.
    Assume that  $E\subseteq\X$ is a bounded set of finite perimeter and let $u_1,\dots,u_n\in\Test(\X)$  be such that 
\begin{equation}\label{eqdeterminant}
    \det\big((\nabla u_i\,\cdot\,\nabla u_j)_{i,j=1,\dots,n})\big)\ge c\qquad\haus^n\text{-a.e.\ on $E$}
\end{equation}
for some $c\in(0,1)$.
Then,  for every $i=1,\dots,n$,
\begin{equation}
    \tr \hess u_i=\Delta u_i\qquad\haus^n\text{-a.e.\ on }E.
\end{equation}
\end{thm}

\begin{proof}[Proof of Theorem \ref{main2} given  Theorem \ref{intermediate}] Fix $u\in\Test(\X)$. Since $L^2(T\X)$ has dimension $n$, we can find a covering $\{|\nabla u|>0\}\subseteq\bigcup_{k\in\NN} E_k$ such that, for every $k$, there exist $u_{2,k},\dots ,u_{n,k}\in\Test(\X)$ such that $\nabla u_{1,k},\dots ,\nabla  u_{n,k}$ form a basis of $L^2(T\X)$ on $E_k$, where $u_{1,k}\defeq u$.
Hence,
\begin{equation}
   \det\big((\nabla u_{i,k}\,\cdot\,\nabla u_{j,k})_{i,j=1,\dots,n})\big)\ne 0\qquad\haus^n\text{-a.e.\ on $E_k$}.
\end{equation}
By the coarea formula, the set     $\Big\{\det\big((\nabla u_{i,k}\,\cdot\,\nabla u_{j,k})_{i,j=1,\dots,n})\big) >s\Big\}$ has finite perimeter for a.e.\ $s$. Hence, up to changing the covering,  we can assume in addition that  $E_k$ has finite perimeter and that
\begin{equation}
   \det\big((\nabla u_{i,k}\,\cdot\,\nabla u_{j,k})_{i,j=1,\dots,n})\big)>c_k\qquad\haus^n\text{-a.e.\ on $E_k$},
\end{equation}
for some $c_k>0$.
 Moreover, as in the proof of Lemma \ref{persing}, we can assume with no loss of generality that $E_k$ is bounded.

Now fix $k$. By Theorem \ref{intermediate}, \begin{equation}
    \tr\hess u=\Delta u\ \qquad\haus^n\text{-a.e.\ on } E_k.
\end{equation}
Since $(E_k)_k$ forms a covering of $\{|\nabla u|>0\}$,  the claim follows on $\{|\nabla u|>0\}$. Lemmas \ref{loc1} and \ref{loc2} conclude the proof on $\{|\nabla u|=0\}$.
\end{proof}

\subsection{Proof of Theorem \ref{intermediate}}
Throughout this section, we work in the setting of the statement of Theorem \ref{intermediate}. Namely, $(\X,\d,\haus^n)$ is as in the statement of Theorem \ref{main2}, and we have a bounded set $E\subseteq\X$  of finite perimeter and  $u_1,\dots,u_n\in\Test(\X)$ satisfying \eqref{eqdeterminant} for some $c\in (0,1)$.

\bigskip

First, we build $n$ regular vector fields $V_1,\dots,V_n$ which form an orthonormal basis on $E$, through a rather classical Gram--Schmidt procedure.
\begin{prop}\label{GS}
   There exist  $V_1,\dots,V_n\in H^{1,2}_H(T\X)\cap L^{\infty}(T\X)$ such that for $i,j=1,\dots,n$, $V_i\,\cdot\,V_j=\delta_{i,j} \ \haus^n$-a.e.\ on  $E$ and $|V_i|\le 1\ \haus^n$-a.e.\ on $\X$.
\end{prop}
\begin{proof}
The proof is a minimal modification of the classical Gram--Schmidt procedure, where the standard case corresponds to  $\eta=0$ below.
We set $L\defeq \max_{i}\|\nabla u_i\|_{L^\infty}$ and then $\eta\defeq\frac{\sqrt{c}}{L^{n-1}}$. By \eqref{eqdeterminant},
\begin{equation}\label{vefdcs}
        \sqrt{c}\le |\nabla u_1\wedge\cdots\wedge\nabla u_n|      \le |\nabla u_1\wedge\cdots\wedge \nabla u_i| L^{n-i}\qquad\haus^n\text{-a.e.\ on $E$}.
\end{equation}

We define
\begin{equation}\label{vfdscs0}
    \tilde V_1\defeq \nabla u_1,\qquad m_1\defeq |\tilde V_1|,\qquad V_1\defeq \frac{\tilde V_1}{m_1\vee \eta},
\end{equation} and, for $i=2,\dots,n$, we define inductively 
\begin{equation}
    \tilde V_i\defeq \nabla u_i-\sum_{j=1}^{i-1}(\nabla u_i\,\cdot\, V_j)V_j,\qquad m_i\defeq|\tilde V_i|,\qquad V_i\defeq\frac{\tilde V_i}{m_i\vee \eta}.
\end{equation}
Regularity of the $V_i$ follows from Lemma \ref{remaininH}. 
Now we want to show by induction that, for every $i=1,\dots,n$,
\begin{equation}\label{vfdscs}
    m_i\ge \eta\quad\haus^n\text{-a.e.\ on $E$}\qquad\text{and}\qquad V_i=\frac{\tilde V_i}{m_i}\quad\haus^n\text{-a.e.\ on $E$},
\end{equation}
which  concludes the proof by the calculus rules.

For $i=1$,  this is an immediate consequence of \eqref{vefdcs}. For the inductive step, take $i=2,\dots,n$. 
As computed in \cite[Proof of Proposition 5.13]{BBPorient}, using the second statement in  \eqref{vfdscs} (for $j=1,\dots,i-1$) we have that 
\begin{equation}
    m_i=\frac{|\nabla u_1\wedge\cdots\wedge\nabla u_i|}{|\nabla u_1\wedge\cdots\wedge\nabla u_{i-1}|}\qquad\haus^n\text{-a.e.\ on $E$}.
\end{equation}
The first statement in  \eqref{vfdscs} follows then from \eqref{vefdcs}, and then also the second one follows immediately by the definition of $V_i$.
\end{proof}
We define the ``volume form''  $\omega\defeq V_1\wedge\cdots \wedge V_n$ and we notice that by Proposition \ref{GS}, we have 
\begin{equation}
    |\omega|=1 \qquad\haus^n\text{-a.e.\ on }E.
\end{equation}
Recall that our goal is to show that $\tr \hess f=\Delta f$, i.e.\ $\tr(\nabla( \nabla f))=\dive (\nabla f)$. We therefore introduce the  deficit operator
\begin{equation}
    \mathscr{D}:H^{1,2}_H(T\X)\rightarrow L^2(\X)\qquad  \mathscr{D}(V)\defeq \tr \nabla V-\dive V
\end{equation}
and our ultimate goal is to prove that it vanishes. The first step is to show that it is linear, in the following strong sense. This will imply, in particular, that $\mathscr{D}$ is continuous.
\begin{prop}\label{vfdssc}
    For every $V\in H^{1,2}_H(T\X)\cap L^\infty(T\X)$,
    \begin{equation}
        \mathscr{D}(V)=\sum_{i=1}^n (V\,\cdot\, V_i ) \mathscr{D}(V_i)\qquad\haus^n\text{-a.e.\ on $E$}.
    \end{equation}
\end{prop}
\begin{proof}
    First, set $\tilde V\defeq \sum_{i=1}^n (V\,\cdot\, V_i) V_i$. Notice that by Proposition \ref{GS} and Lemma \ref{remaininH}, $\tilde V\in H^{1,2}_H(T\X)\cap L^\infty(T\X)$ and  $\tilde V=V$ $\haus^n$-a.e.\ on $E$. By locality, see Lemmas \ref{loc1} and \ref{loc2}, $\mathscr{D}(V)=\mathscr{D}(\tilde V)$ $\haus^n$-a.e.\ on $E$. Moreover, if $f\in H^{1,2}(\X)\cap L^\infty(\X)$,  the following holds $\haus^n$-a.e.\ on $E$:
    \begin{equation}
    \begin{split}
        \mathscr{D}(fV_i)&=\tr(\nabla(f V_i))-\dive(f V_i)=\tr (f\nabla V_i)+\tr(\nabla f\otimes V_i)- \nabla f\,\cdot\, V_i-f\dive V_i\\
        &=f\tr \nabla V_i+\nabla f\,\cdot\,V_i-\nabla f\,\cdot\, V_i-f\dive V_i=f\mathscr{D}(V_i).
    \end{split}
    \end{equation}
    The conclusion then follows from the linearity of $\mathscr{D}$.
\end{proof}

Now we define  the following functional:
\begin{equation}\label{defnt}
    T:\LIP(\X)^{n+1}\rightarrow\RR\qquad T(f,\pi_1,\dots,\pi_n)\defeq \int_E fd\pi_1\wedge\cdots\wedge d\pi_{n}\,\cdot\, \omega \de \haus^n.
\end{equation}
It follows from \cite[Proposition 6.2]{BBPorient} (see also \cite[Equation (6.2)]{BBPorient}) that $T$ is a metric current satisfying
\begin{equation}\label{massofT}
    |T|=\haus^n\mres E.
\end{equation}
To be more precise, \cite{BBPorient} deals with finite-dimensional $\RCD$ spaces; however, it is clear  from the proof that the results used above extend to  infinite-dimensional   $\RCD$ spaces, under the additional assumption that the tangent module has constant  dimension $n$.

Now, let $f,\pi_1,\dots,\pi_{n-1}\in\Test(\X)$. Proposition \ref{fdcsc} associates to
\begin{equation}\label{aaavefdcsc}
    \eta\defeq fd\pi_1\wedge\cdots\wedge d\pi_{n-1}
\end{equation} the vector field $\alpha_\eta\in H^{1,2}_H(T\X)$. By Proposition \ref{fdcsc}, $d\eta=\tr(\nabla\alpha_\eta)\omega\ \haus^n$-a.e.\ on $E$. We can then compute
     \begin{equation}\label{partialT}
    \begin{split}
        \partial T(f,\pi_1,\dots,\pi_{n-1})&=T(1,f,\pi_1,\dots,\pi_{n-1})=\int_E d f\wedge d\pi_1\wedge\cdots\wedge d\pi_{n-1}\,\cdot\,\omega\de\haus^n\\
        &=\int_E d\eta\,\cdot\,\omega\de\haus^n=\int_E \tr(\nabla\alpha_\eta)\omega\,\cdot\,\omega \de\haus^n= \int_E \tr \nabla\alpha_\eta \de\haus^n\\
        &= \int_E \mathscr{D}( \alpha_\eta)\de\haus^n+\int_E \dive \alpha_\eta \de\haus^n.
        \end{split}
\end{equation}

The computation above is the main point of the proof of the main result. It shows that the boundary of $T$ is the sum of two terms: The first term involves the deficit $\mathscr{D}$, while the second  can be integrated by parts and can thus be estimated by an integral with respect to a perimeter. This will allow us to easily show that $\partial T$ has finite mass. 

Proposition \ref{Tint} below proves that $T$ is integer-rectifiable. Once this is established, the Ambrosio--Kirchheim boundary-rectifiability theorem (Theorem \ref{boundrect}) will imply that the  first term on the right-hand side of \eqref{partialT} must vanish, see  Proposition \ref{vcdscscd} below.

\begin{prop}\label{Tint}
    It holds that $T\in \mathbf{I}_n(\X)$.
\end{prop}
\begin{proof}
    We first prove that $T\in\mathbf{N}_n(\X)$. Take $f,\pi_1,\dots,\pi_{n-1}\in\Test(\X)$. We want to estimate the boundary $\partial T(f,\pi_1,\dots,\pi_{n-1})$.   Let $\eta$ and $\alpha_\eta$ be as above.
    By the definition of $\alpha_\eta$ in Proposition \ref{fdcsc},  we have the rough  (on $E$, a more precise estimate is available) estimate
    \begin{equation}
        |\alpha_\eta|\le n |f|\Lip(\pi_1)\cdots\Lip(\pi_{n-1})\qquad\haus^n\text{-a.e.}
    \end{equation}
    Now we estimate the two terms on the right-hand side of \eqref{partialT} separately. 
    For the first term, by Proposition \ref{vfdssc},
    \begin{equation}
        \begin{split}
            \Big| \int_E \mathscr{D}( \alpha_\eta)\de\haus^n   \Big|&\le \int_E \sum_{i=1}^n |\alpha_\eta\,\cdot\, V_i| |\mathscr{D}(V_i)|\de\haus^n \\
            &\le  n \Lip(\pi_1)\cdots\Lip(\pi_{n-1})\int_E |f|\sum_{i=1}^n |\mathscr{D}(V_i)|\de\haus^n.
        \end{split}
    \end{equation}
    For the second term, integrating by parts (\cite{BGBV}, after \cite{bru2019rectifiability}),
    \begin{equation}\label{inteparts}
        \Big|\int_E \dive\alpha_\eta\de\haus^n\Big|=\Big|\int\alpha_\eta\,\cdot\,\nu_E\de|D\chi_E|\Big|\le n\Lip(\pi_1)\cdots\Lip(\pi_{n-1})\int|f|\de|D\chi_E|,
    \end{equation}
    where we used that $|f|$ is continuous when taking the trace of $|\alpha_\eta|$ for the measure $|D\chi_E|$.
    
    All in all
    \begin{equation}
        \big|\partial T(f,\pi_1,\dots,\pi_{n-1})\big|\le   n\Lip(\pi_1)\cdots\Lip(\pi_{n-1})\Big(\int_E |f|\sum_{i=1}^n |\mathscr{D}(V_i)|\de\haus^n+\int|f|\de|D\chi_E|\Big).
    \end{equation}
    By the continuity axioms for currents and the regularizing properties of the heat flow on $\RCD(K,\infty)$ spaces,  the above continues to hold whenever $f,\pi_1,\dots,\pi_{n-1}\in\LIP(\X)$. 
    This means that $\partial T$ has finite mass, more precisely,
    \begin{equation}
        |\partial T|\le n\Big(\sum_{i=1}^n |\mathscr{D}(V_i)|\haus^n\mres E+|D\chi_E|\Big).
    \end{equation}
    In passing, we notice that \cite[Theorem 8.7]{AmbrosioKirchheim00} would imply that $E$ must be $n$-rectifiable, even without using the rectifiability of $\X$ at this stage of the proof (cf.\ Remark \ref{rectused}).

    Finally, recall that by Proposition \ref{maincurrent} with \eqref{massofT}, $T$ is also integer-rectifiable.
    \end{proof}
    
    \begin{prop}\label{vcdscscd}
        It holds that $\mathscr{D}(V_i)= 0\ \haus^n$-a.e.\ on $E$, for every $i=1,\dots,n$.
    \end{prop}
\begin{proof} 
    Fix for the moment $\pi_1,\dots,\pi_{n-1}\in\Test(\X)$. For any $f\in\Test(\X)$, we consider $\eta=fd\pi_1\wedge\cdots\wedge d\pi_{n-1}$ as in  \eqref{aaavefdcsc}. 
    Proposition \ref{fdcsc} gives a vector field $\alpha_\eta \in H^{1,2}_H(T\X)\cap L^\infty(T\X)$. 
    
      By Proposition \ref{Tint}, we can apply Theorem \ref{boundrect} and obtain that $\partial T\in\mathbf{I}_{n-1}(\X)$. 
    Hence, $|\partial T|$ is concentrated on a countably $\haus^{n-1}$-rectifiable set and vanishes on $\haus^{n-1}$-negligible sets. 
    In particular, there exists a Borel set $B\subseteq \X$ with $\haus^n(B)=0$ such that $|\partial T|=|\partial T|\mres B$, hence
    \begin{equation}
        |\partial T(f,\pi_1,\dots,\pi_{n-1})|\le \Lip(\pi_1)\cdots\Lip(\pi_{n-1})\int_{B} |f|\de |\partial T|.
    \end{equation}
    By \eqref{partialT} and the integration by parts formula recalled in \eqref{inteparts}, we see that
    \begin{equation}
        \Big|\int_E \mathscr{D}(\alpha_{ \eta})\de \haus^n\Big|\le \Lip(\pi_1)\cdots\Lip(\pi_{n-1})\Big(n\int|f|\de |D\chi_E|+\int_{B} |f|\de |\partial T|\Big).
    \end{equation}
    Now recall that $\Test(\X)$ is an algebra, so that, for any $g\in\Test(\X)$, we can replace $f$ by $gf$ in the argument above. In particular, using also  Proposition \ref{vfdssc} and recalling  \eqref{defalpha}, we see that 
\begin{equation}
        \Big|\int_E g\mathscr{D}(\alpha_{ \eta})\de \haus^n\Big|\le \Lip(\pi_1)\cdots\Lip(\pi_{n-1})\Big(n\int|fg|\de |D\chi_E|+\int_{B} |fg|\de |\partial T|\Big).
    \end{equation}
   Since $g\in\Test(\X)$  is arbitrary,  as $B$ is $\haus^n$-null and $|D\chi_E|\perp\haus^n$ (by Lemma \ref{persing}), an approximation argument yields that
    \begin{equation}
      \mathscr   D(\alpha_{\eta})=0\qquad\haus^n\text{-a.e.\ on }E.
    \end{equation}
    Using again Proposition \ref{vfdssc} and \eqref{defalpha}, being  $f\in\Test(\X)$ arbitrary, we see that 
    \begin{equation}
      \sum_{i=1}^n(-1)^{i-1}\big( d\pi_1\wedge\cdots\wedge d\pi_{n-1}\,\cdot\,V_1\wedge\cdots \wedge \hat V_i \wedge\cdots \wedge V_n\big)\mathscr{D}(V_i)=0  \qquad\haus^n\text{-a.e.\ on }E.
    \end{equation}
    As $\pi_1,\dots,\pi_{n-1}\in\Test(\X)$ were arbitrary, by a density argument we see that the above implies
    \begin{equation}
        \mathscr{D}(V_i)=0  \qquad\haus^n\text{-a.e.\ on }E
    \end{equation}
    for every $i=1,\dots,n$.
\end{proof}

\begin{proof}[Proof of Theorem \ref{intermediate}]
    Proposition \ref{vcdscscd} implies  that $\mathscr{D}(V_{i})=0\ \haus^n$-a.e.\ on $E$, for every $i=1,\dots,n$.
By the definition of $V_{i}$ in  Proposition \ref{GS}, i.e., \eqref{vfdscs0},  we see that $V_{1}=\frac{\nabla u_1}{|\nabla u_1|\vee \eta}$, for an appropriate $\eta>0$. Proposition \ref{vfdssc} then implies that $\mathscr{D}(\nabla u_1)=0\ \haus^n$-a.e.\ on $E$. This, of course, holds for the other indices as well.
\end{proof}

\end{document}